\documentclass[11pt]{amsart}
\usepackage{amsmath,amsthm,epsfig,amssymb}
\usepackage{amsthm,amsfonts,amssymb,amscd}
\usepackage[utf8x]{inputenc}
\usepackage{textcomp}
\usepackage{tikz}
\usetikzlibrary{positioning, matrix, arrows}
\usepackage[all]{xy}
\usepackage[colorlinks=true,citecolor=blue,urlcolor=blue,linkcolor=blue]{hyperref}
\usepackage{url}

\newtheorem{Theorem}[subsection]{Theorem}
\newtheorem{Proposition}[subsection]{Proposition}

\newtheorem{Lemma}[subsection]{Lemma}

\theoremstyle{definition}
\newtheorem{Remark}[subsection]{Remark}
\newtheorem{Definition}[subsection]{Definition}

\newtheorem{proposition-definition}[subsection]{Proposition-Definition}

\newcommand{\M}{\mathcal{M}}

\title{Non-emptiness of Brill-Noether Loci over very general quintic 
hypersurface}
\author[K. Dan]{Krishanu Dan}
\address{Chennai Mathematical Institute, H1 Sipcot IT Park, 
Siruseri, Kelambakkam - 603103, INDIA.}
\email{krishanud@cmi.ac.in}

\author[S.Pal]{Sarbeswar Pal}
\address{IISER - Thiruvananthapuram, Computer Science Building,
College of Engineering Trivandrum Campus,
Trivandrum - 695016, Kerala, India}
\email{spal@iisertvm.ac.in}


\begin{document}

\begin{abstract}
 In this article we study Brill-Noether loci of moduli space of stable bundles 
 over smooth surfaces. We define Petri map as an analogy with the case of curves.
 We show the non-emptiness of certain Brill-Noether loci 
 over very general quintic hypersurface in $\mathbb{P}^3$, and use the 
 Petri map to produce components of expected dimension.

\end{abstract}

\maketitle

\section{Introduction}

Let $X$ be a smooth, irreducible, projective variety of dimension $n$ over $\mathbb{C}$, 
$H$ be an 
ample divisor on $X$, and let $\M:= \M_{X,H}(r; c_1, \cdots ,c_s)$ be the moduli 
space of rank $r$, $H$-stable vector bundles $E$ over $X$ with Chern classes 
$c_i(E)=c_i,$ where $s:= \text{min}\{r,n\}$. A Brill-Noether locus $B_{r,X,H}^k$ is a closed 
subscheme of $\M$ whose support consists of points $E\in \M$ such that $h^0(X,E) 
\geq k+1$. G\"{o}ttsche et al (\cite{GH}) and M. He (\cite{He}) 
studied the Brill-Noether loci of stable 
bundles over $\mathbb{P}^2$, and Yoshioka (\cite{Yoshioka1}, \cite{Yoshioka2}), 
Markman (\cite{Markman}), Leyenson (\cite{Leyenson1}, \cite{Leyenson2}) 
studied it for $K3$ surfaces.  In the case of smooth, projective, irreducible 
curves $C$ over $\mathbb{C}$, the 
Brill-Noether loci of the moduli space, $\text{Pic}^d(C)$, of degree $d$ line bundles 
on $C$ is well-studied. The questions like non-emptiness, connectedness, 
irreducibility, singular locus etc of Brill-Noether loci are known when $C$ is 
a general curve in the sense of moduli (see e.g. \cite{ACGH}). This concept was 
generalized for vector bundles over curves by Newstead, Teixidor and others. For an 
account of the results and history in this case see \cite{GB} and the references 
therein. 

Recently, in \cite{CM1}, authors have constructed a 
Brill-Noether loci over higher dimensional varieties under the additional 
cohomology vanishing assumptions: $H^i(X,E)=0, \forall i \geq 2$ and for all $E \in \M$. 
This is a natural generalization of the Brill-Noether loci over the 
curves for higher dimensional varieties. In \cite{CM1}, \cite{CM2}, authors gave 
several examples of non-empty Brill-Noether locus, and examples of Brill-Noether 
locus where ``expected dimension'' is not same as the exact dimension. In 
all these examples, the surfaces and/or varieties chosen have the canonical 
bundle has no non-zero sections.
In this article, we defined ``Petri map'' over a smooth projective variety with 
canonical bundle ample,  as an analogue of that for curves. Similar 
to the case of curves, the injectivity of the ``Petri map'' implies the existence 
of smooth points in Brill-Noether loci.
We then use this fact to prove the existence of a smooth point and hence a component of 
expected dimension in the Brill-Noether loci  over a very general quintic hypersurface in 
$\mathbb{P}^3$ where the canonical bundle is ample and globally generated.  

{\it Notation:} We work throughout over the field $\mathbb{C}$ of complex numbers. 
If $X$ is a smooth, projective variety, we denote by $K_X$ the canonical bundle on $X$. 
For a coherent sheaf $\mathcal{F}$ on $X$, we denote by $H^i(X,\mathcal{F})$ the 
$i$-th cohomology group of $\mathcal{F}$ and by $h^i(X,\mathcal{F})$ its (complex) 
dimension. If $V$ is a vector bundle on $X$, we denote by $V^*$ the dual of $V$.

\section{Brill-Noether Loci}

In this section we will briefly recall the construction of Brill-Noether loci over 
higher dimensional varieties, following \cite{CM1}.

Let $X$ be an irreducible, smooth, projective variety of dimension $n$, and let 
$H$ be an ample divisor on $X$. For a torsion free sheaf $F$ over $X$, let 
$c_i(F)$ denotes the $i$-th Chern class of $F$. Set $\mu(F) = \mu_H(F) := 
\frac{c_1(F)\cdot H^{n-1}}{\text{rank}(F)}$. 

\begin{Definition}
 A torsion-free sheaf $F$ over $X$ of rank $r$ is called $H$-{\it semistable} if for 
 all non-zero subsheaf $G$ of $F$ with ${\rm{rank}}(G) < {\rm{rank}}(F)$, we have 
 $$
 \mu(G) \leq \mu(F).
 $$
 We say $F$ is $H$-{\it stable} if the above inequality is strict.
\end{Definition}

Let $\M:= \M_{_{X,H}}(r; c_1, \cdots ,c_s)$ be the moduli space of rank $r$, 
$H$-stable vector bundles $E$ over $X$ with Chern classes $c_i(E)=c_i,$ where 
$s:= \text{min}\{r,n\}$. Assume that $\M$ is a fine moduli space, and let 
$\mathcal{E} \to \M \times X$ be an universal family such that for any 
$t \in \M, E_t:= \mathcal{E}|_{_{t \times X}}$ is a rank $r, H$-stable bundle 
over $X$ with Chern classes $c_i(E_t)=c_i$. Choose an effective divisor $D$ 
on $X$ such that $H^i(X, E_t(D))=0, \forall i \geq 1$ and $\forall t \in M$. 
Let $\mathcal{D}:= \M \times D$ be the product divisor, and $\phi: \M \times X 
\to \M$ be the projection map. From the exact sequence
$$
0 \to \mathcal{E} \to \mathcal{E}(\mathcal{D}) \to 
\mathcal{E}(\mathcal{D})/\mathcal{E} \to 0
$$
on $\M \times X$, we get an exact sequence on $\M$:
$$
0 \to \phi_{_*}(\mathcal{E}) \to \phi_{_*}(\mathcal{E}(\mathcal{D})) 
\xrightarrow{\gamma} \phi_{_*}(\mathcal{E}(\mathcal{D})/\mathcal{E}) \to 
R^1\phi_{_*}(\mathcal{E}) \to 0.
$$
Note that, $\gamma$ is a map between two locally free sheaves of ranks 
$\chi(E_t(D))$ and $\chi(E_t(D)) - \chi(E_t)$ respectively, on $\M$. For an 
integer $k \geq -1$, let $B_{r, X, H}^k \subset \M$ be the $(\chi(E_t(D)) - 
(k+1))$-th determinantal variety associated to the map $\gamma$. Now assume 
$H^i(X, E_t)=0, \forall i \geq 2$ and $\forall t \in \M$. Then we have
$$
{\rm{Support}}(B_{r, X, H}^k) = \{E \in \M : h^0(X,E) \geq k+1 \}.
$$
When $\M$ is not a fine moduli space, it is possible to carry out this 
construction locally and then can be glued together to get a global algebraic 
object. We summarize the above construction as
\begin{Theorem}(\cite{CM1}, Theorem 2.3)\label{T1}
Let $X$ be a smooth, irreducible, projective variety of dimension $n$, 
$H$ be a fixed ample divisor on $X$, and  
$\M:=\M_{_{X, H}}(r; c_1, \cdots, c_s)$ be a moduli space of rank $r, H$-stable vector 
bundles $E$ on $X$ with fixed Chern classes $c_i(E)=c_i, s= \text{min}\{r, n\}.$ 
Assume that for any $E \in \M, H^i(X, E)=0$ for $i \geq 2.$ Then for any $k \geq -1,$ 
there exists a determinantal variety $B_{r,X,H}^k \subset \mathcal{M}$ such that 
$$
\text{Support}(B_{r,X,H}^k) = \{E \in \mathcal{M} \,\, : \,\, h^0(X, E) \geq k+1\}.
$$
Moreover, each non-empty irreducible component of $B_{r,X,H}^k$ has dimension at least 
$$
\text{dim}(\mathcal{M})-(k+1)(k+1 - \chi(r; c_1, \cdots , c_s) )
$$
where $\chi(r; c_1, \cdots , c_s):= \chi(E_t)$ for any $t \in \mathcal{M}$ and 
$$
B_{r,X,H}^{k+1} \subset \text{Sing}(B_{r,X,H}^k)
$$
whenever $B_{r,X,H}^k \neq \mathcal{M}.$
\end{Theorem}

\begin{Definition}
 The variety $B_{r,X,H}^k$ is called the {\it $k$-th Brill-Noether locus} of the moduli 
 space $\M$ and the number
 $$
 \rho_{r,X,H}^k := \text{dim}(\mathcal{M})-(k+1)(k+1 - \chi(r; c_1, \cdots , c_s) )
 $$
 is called the {\it generalized Brill-Noether number.}
\end{Definition}

By the above theorem, the dimension of $B_{r,X,H}^k$ is at least $\rho_{r,X,H}^k$. We 
call $\rho_{r,X,H}^k$, the {\it expected dimension} of the Brill-Noether locus $B_{r,X,H}^k$. 
When there is no confusion about $X$ and $H$, we will simply denote these by 
$B_r^k$ and $\rho_r^k$.

\section{Petri Map}

In this section we will define ``Petri Map'' for higher dimensional varieties, as an 
analogue of the one defined for curves. We note that the description of Petri map 
over curves as given in \cite{GB} works for higher dimensional varieties also. For 
convenience, we recall this description.

Let $X$ be an irreducible, smooth, projective variety of dimension $n$, 
$H$ be an ample divisor on $X$, and let 
$\M:= \M_{_{X,H}}(r; c_1, \cdots ,c_s)$ be the moduli space of rank $r$, 
$H$-stable vector bundles $E$ over $X$ with Chern classes $c_i(E)=c_i,$ where 
$s:= \text{min}\{r,n\}$. The tangent space to $\M$ at a point $E\in \M$ is 
given by $H^1(X, E \otimes E^*)$, where $E^*$ is the dual of $E$. A tangent 
vector to $\M$ at $E$ can be identified with a vector bundle $E_{\epsilon}$ on 
$X_{\epsilon}:= X \times \text{Spec}(k[\epsilon]/\epsilon^2)$ whose restriction 
on $X$ is $E$ and it fits into the exact sequence
$$
0 \to E \to E_{\epsilon} \to E \to 0.
$$
We call it a first order deformation of $E$.\\
One can give an explicit description of the bundle $E_{\epsilon}$ as follows: 
Let $\{U_i\}$ be an open cover of $X$ such that $E_i:= E|_{U_i}$ is the trivial 
bundle. Set $U_{ij}:= U_i \cap U_j$, and let $\phi_{ij} \in H^0(U_{ij}, E\otimes E^*)$ 
be the co-boundary map corresponding to $\phi \in H^1(X, E \otimes E^*)$. Consider 
the trivial extension of $E_i$ to $U_i \times \text{Spec}(k[\epsilon]/\epsilon^2)$ 
given by $E_i \oplus \epsilon E_i$. Then the matrix
$$
  \begin{bmatrix}
    \text{Id} & 0  \\
    \phi_{ij} & \text{Id}
    \end{bmatrix}
$$
will give the gluing data for the bundle $E_{\epsilon}$.

Assume that a section $s$ of $E$ can be extended to a section of $E_{\epsilon}$. Then we 
have local sections $(s_i') \in H^0(U_i, E_i)$ such that $(s|_{U_i}, s_i')$ 
defines a section of $E_{\epsilon}$. If this is the case, then we have
$$
  \begin{bmatrix}
    \text{Id} & 0  \\
    \phi_{ij} & \text{Id}
    \end{bmatrix}
    \begin{bmatrix}
    s|_{U_i}  \\
    s_i'
    \end{bmatrix}
    =
    \begin{bmatrix}
    s|_{U_j}  \\
    s_j'
    \end{bmatrix}.
$$
This gives two conditions: $(s|_{U_i})|_{U_{ij}}= (s|_{U_j})|_{U_{ij}}$ and 
$\phi_{ij}(s)=s_j' - s_i'$. The first condition is automatically satisfied, since 
$s$ is a global section and from the second condition we see that, in this case, 
$(\phi_{ij}(s))$ satisfies the co-cycle condition. In other words,  
$(\phi_{ij}(s))$ is in the kernel of the map
$$
H^1(X, E \otimes E^*) \longrightarrow H^1(X, E),  \hspace{0.3 cm} 
(\nu_{ij}) \mapsto (\nu_{ij}(s)).
$$

Let $E \in B_r^k, k \geq 0$ and let $s \in H^0(X, E)$. Then the first order 
deformation of $E$, as an element of $B_r^k$, is the subset $\{E_{\epsilon} : $
the section $s$ can be extended to a section of $E_{\epsilon}\}$ 
of $H^1(X, E \otimes E^*)$, i.e. 
$$
E_{\epsilon} \in \text{Ker}\big(H^1(X, E \otimes E^*) \longrightarrow H^1(X, E),  
\hspace{0.3 cm} (\nu_{ij}) \mapsto (\nu_{ij}(s))\big).
$$
Now assume $E \in B_r^k - B_r^{k+1}$ and let $T$ be the tangent space to $B_r^k$ 
at the point $E$. From the discussion above, we have a map 
$$
\alpha: H^0(X,E) \otimes H^1(X, E \otimes E^*) \longrightarrow H^1(X,E).
$$
This induces the map
$$
\mu: H^0(X,E) \otimes H^{n-1}(X, K_X \otimes E^*) \longrightarrow 
H^{n-1}(X, K_X \otimes E \otimes E^*).
$$
Note that, $T$ can be identified 
with $(\text{Im}(\mu))^{\bot}$. We call the map $\mu$, the {\it Petri map.}

\begin{Remark}
 Let $X$ be an irreducible, smooth, projective surface and $H$ be an ample divisor 
 on $X$. In this case, Petri map, as defined above, is the cup product map
 \begin{equation}\label{eqn 1}
  \mu: H^0(X,E) \otimes H^{1}(X, K_X \otimes E^*) \longrightarrow 
  H^{1}(X, K_X \otimes E \otimes E^*).
 \end{equation}
 If $E$ is a smooth point in the moduli space $\M$, we have
 $$
 \text{dim}(T)= \text{dim}(\mathcal{M}) - h^0(X, E)h^1(X, E) + \text{dim Ker}(\mu).
 $$
 Thus, if the Petri map is injective, then $E$ is a smooth point of $B_r^k$ and 
 and the component of $B^k_r$ through $E$ has the expected dimension.
\end{Remark}

\begin{Remark}\label{rem1}
 The Petri map can also be derived from \cite{He}. 
 Indeed, by taking $\Lambda=\Lambda'=(H^0(X,E),Id,E)$ in \cite[Corollary 1.6]{He}, 
 we see that the above Petri map is dual of the map $Ext^1(E,E) \to 
 Hom(H^0(X,E),H^1(X,E))$ in the given exact sequence.
\end{Remark}

\section{Brill-Noether loci over quintic hypersurface}
 
Let $X$ be a very general quintic hypersurface in $\mathbb{P}^3$. Then we have 
$\text{Pic}(X) \simeq \text{Pic}(\mathbb{P}^3)\simeq \mathbb{Z}$, $K_X \simeq 
\mathcal{O}_X(1)$, $\chi(X,\mathcal{O}_X)=5$. Let $H$ be a hyperplane class, and 
$\M(c_2):=\M_{X,H}(2; 3H, c_2)$ be the moduli space of rank two $H$-stable bundles 
on $X$ with first Chern class $3H$, and second Chern class $c_2$.
It is known (\cite{Simpson1}) that $\M(c_2)$ is irreducible 
for $c_2 \geq 14$, generically smooth for $c_2 \geq 21$, and has the expected 
dimension $4c_2 -60$. Also note that $H^2(X,E)=0, \forall E \in \M$. Thus the 
hypothesis of the Theorem \ref{T1} is satisfied.

Let $C$ be a smooth, irreducible, projective curve in the complete linear system $|3H|$. 
Then the genus of $C, g:= 31$.

\begin{Proposition}\label{P4.1}
 With the notations as above, assume that there is a base point free line bundle 
 of degree $\geq 16$ on $C$ with exactly two sections, then $B^1_2 \subset \M(c_2)$ 
 is non-empty.
\end{Proposition}
\begin{proof}
 Let $A$ be a base point free line bundle on $C$ with $h^0(C,A)=2$. Consider the 
 elementary transformation
 \begin{equation}\label{eqn4.1}
  0 \to F \to H^0(C,A)\otimes \mathcal{O}_X \to A \to 0.
 \end{equation}
 Then $F$ is a rank two vector bundle on $X$ with $c_1(F)=-3H$, $c_2(F)=\text{deg}(A)$, 
 $h^0(X,F)=0=h^1(X,F)$ (\cite[Chapter 5, Proposition 5.2.2]{HL}). Dualizing the 
 above exact sequence we get
 \begin{equation}\label{eqn4.2}
  0 \to H^0(C,A)^*\otimes \mathcal{O}_X \to F^* \to \mathcal{O}_C(C)\otimes A^* \to 0.
 \end{equation}
 Thus $h^0(X,F^*)=2+h^0(C, \mathcal{O}_C(C)\otimes A^*) \geq 2$.\\
 {\it Claim:} $F^*\in B^1_2$.\\
 It is sufficient to show that $F$ is $H$-stable. Let $\mathcal{O}_X(m)$ destabilizes 
 $F$. Then $m \geq -1$. On the other hand, from (\ref{eqn4.1}), we have $m \leq 0$. 
 Since $h^0(X,F)= 0$, $m \neq 0$. Thus we are reduced to show that 
 $h^0(X, F \otimes \mathcal{O}_X(1))=0$. Note that $F \otimes \mathcal{O}_X(1) \simeq 
 F^* \otimes \mathcal{O}_X(-2)$. Tensoring the exact sequence (\ref{eqn4.2}) by 
 $\mathcal{O}_X(-2)$ we get
 $$
 0 \to \mathcal{O}_X(-2)^{\oplus 2} \to F^*\otimes \mathcal{O}_X(-2) \to 
 \mathcal{O}_C(H)\otimes A^* \to 0.
 $$
 Since $\text{deg}(\mathcal{O}_C(H)\otimes A^*) < 0$, $h^0(C, \mathcal{O}_C(H)\otimes A^*) 
 =0$ and consequently $h^0(X, F^*\otimes \mathcal{O}_X(-2))=0.$
\end{proof}

Since $C \in |3H|$, $C$ is a smooth complete intersection of two smooth hypersurfaces 
of $\mathbb{P}^3$ of degrees $5$ and $3$. Thus by \cite[Page 139, C-4]{ACGH}, $C$ does 
not have a $g^1_5$. In particular, $C$ is not hyperelliptic, trigonal or tetragonal 
(i.e. $C$ does not have a $g^1_2, g^1_3, g^1_4$ respectively). Also note that 
$K_C = \mathcal{O}_C(4)$. Now for $D=\mathcal{O}_C(3)$, $K_C \otimes \mathcal{O}_C(-D) 
= \mathcal{O}_C(1)$, and this gives an embedding of $C \hookrightarrow \mathbb{P}^3$. 
So by \cite[Page 221, B-4]{ACGH}, $C$ is not a bi-elliptic.

\begin{Proposition}\label{P4.2}
 With the notations as above, for $28 \leq d \leq 32$, there exists a base point free 
 line bundle of degree $d$ on $C$. 
\end{Proposition}
\begin{proof}
 Let us denote by $W^r_d(C)$ the Brill-Noether loci of degree $d$ line bundles $L$ on $C$ 
 with $h^0(C,L)\geq r+1$.\\
 {\it Case I:} $d=32$.\\
 In this case, $\text{dim}W^1_{32}(C)=31$. If $W^1_{32}(C)-W^2_{32}(C) (\neq \varnothing)$ 
 does not contain any base point free line bundle, then tensoring by the ideal sheaf of 
 the base locus, we obtain a family of base-point free line bundles with exactly two 
 sections of dimension $31$ and is contained in $\bigcup_{e\leq 31}W^1_e(C)$. Now 
 $\text{dim}W^0_e(C)=e \leq 31=g$, and $W^1_e(C)$ is a proper closed subset of $W^0_e(C)$. So 
 $\text{dim}W^1_e(C)\leq e-1 \leq 30$. Thus $\text{dim}(\bigcup_{e\leq 31}W^1_e(C)) < 31$, 
 a contradiction.\\
 {\it Case II:} $d=31$.\\
 In this case, $\text{dim}W^1_{31}(C) \geq 29$.  Applying Martens' theorem 
 \cite[Chapter IV, Theorem 5.1]{ACGH}, 
 we see that $\text{dim}(\bigcup_{2\leq e \leq 30}W^1_e(C))<29$. Note that, for any $L \in \text{Pic}(C)$ 
 with $h^0(C,L)=2$, if $B$ is the base locus of $|L|$, then $\text{deg}(B)\leq \text{deg}(L) -2$. 
 Now arguing as above, we get a base point free line bundle on $C$ of degree $31$ with 
 exactly two sections.\\
 {\it Case III:} $d=30$.\\
 This follows from the existence of base point free complete $g^1_{30}$ on $C$ 
 \cite[Theorem 1.4]{Keem}.\\
 {\it Case IV:} $d=29$.\\
 We have $\text{dim}W^1_{29}(C)\geq 25$. By Mumford's theorem \cite[Chapter IV, Theorem 5.2]{ACGH}, we get 
 $\text{dim}(\bigcup_{2\leq e \leq 28}W^1_e(C))<25$. Now we argue as in Case I to conclude.\\
 {\it Case V:} $d=28$.\\
 Here $\text{dim}W^1_{28}(C)\geq 23$. Then using Keem's theorem \cite[Page 200]{ACGH}, we get 
 $\text{dim}(\bigcup_{e \leq 27}W^1_e(C))<23$. Arguing as in Case I, we conclude that there 
 exists a base point free line bundle of degree $28$ on $C$ with exactly two global sections.
\end{proof}

Now we are ready to prove
\begin{Theorem}\label{T4.1}
 With the notations as in the beginning of this section, $B^1_2-B^2_2 \subset \M(c_2)$ 
 is non-empty for $28 \leq c_2 \leq 32$.
\end{Theorem}
\begin{proof}
 Combining Propositions \ref{P4.1} and \ref{P4.2}, we see that, for $28 \leq c_2 \leq 32$, 
 $B^1_2 \subset \M(c_2)$ is non-empty. Moreover, for $28 \leq d \leq 32$, we can find 
 base point free line bundles $A$ on $C$ with $h^0(C,A)=2$ and $h^0(C,\mathcal{O}_C(C)\otimes A^*)=0$. 
 Now using (\ref{eqn4.2}), we conclude.
\end{proof}

\begin{Proposition}\label{P1}\label{P4.3}
 With the notations as in the beginning of this section, let $E \in B_2^1-B_2^2 
 \subset \M(c_2), 28 \leq c_2 \leq 32$ 
 be a general element constructed as in Proposition \ref{P4.1}. Then $E$ is a 
 smooth point in $\M(c_2)$.
\end{Proposition}

\begin{proof}
 Since the morphism $\M_{X,H}(2; 3H, c_2) \xrightarrow{\otimes \mathcal{O}_X(-1)} 
 \M_{X,H}(2; H, c_2-10)$ is an isomorphism, it is sufficient to prove that 
 $E \otimes \mathcal{O}_X(-1)$ is a smooth point of $\M_{X,H}(2; H, c_2-10)$. 
 Set $F:= E \otimes \mathcal{O}_X(-1)$. If $F$ is not a smooth point of 
 $\M_{X,H}(2; H, c_2-10)$, we have $H^0(X, AdF \otimes K_X) \neq 0$. Thus there 
 is a non-zero map $\phi: F \to F \otimes K_X$. Two cases can occur:\\
 Case I: $\phi$ drops rank every where.\\
 Case II: $\phi$ does not drop rank every where.\\
 Following \cite{Simpson2}, we call the Case I as singularity of first kind, and 
 the Case II as singularity of second kind. If Case I occurs, we will have 
 $h^0(X, F) \neq 0$ (see \cite[Proposition 5.1]{Simpson2}), which is impossible. Also, from 
 \cite[Lemma 10.1]{Simpson1}, we see that the dimension of the singular locus of the second 
 kind is at most $8$ \footnotemark. 
\footnotetext{In \cite[Corollary 5.1]{Simpson2}, the authors have estimated the bound for 
singularity of second kind as $\leq 13$. But in \cite[Lemma 10.1]{Simpson1}, the authors have 
improved the above bound and shown that it is $\leq 8$}
 But the dimension of the family of globally generated line bundles $A$ 
 on $C$ with $28 \leq \text{deg}(A) \leq 32, h^0(C,A)=2$ and $h^0(C, \mathcal{O}_C(C) 
 \otimes A^*)=0$ is greater than $8$. Hence a general $E$, as constructed in 
 Proposition \ref{P4.1}, will be a smooth point in $\M_{X,H}(2; 3H, c_2)$.
\end{proof}

Let $E \in B_2^1-B_2^2$ be a smooth point in $\M(c_2), 28 \leq c_2 \leq 32,$ as constructed 
in Proposition \ref{P4.1}. Existence of such a smooth point is assured by Proposition \ref{P4.3}. 
Then $E$ fits into the following exact sequence
\begin{equation}\label{eqn3}
 0 \to \mathcal{O}_X^{\oplus 2} \to E \to \mathcal{O}_C(C) \otimes A^* \to 0
\end{equation}
(see equation (\ref{eqn4.2})). Tensoring this exact sequence by $E^* \otimes K_X$ we get 
\begin{equation}\label{eqn4}
 0 \to (E^*\otimes K_X)^{\oplus 2} \to \mathcal{E}nd(E)\otimes K_X \to 
  \mathcal{O}_C(C) \otimes A^* \otimes (E^*\otimes K_X)|_{_C} \to 0.
\end{equation}
Note that $h^0(X, E^* \otimes K_X)=h^2(X,E)=0$, by our assumption, and
since $E$ is a smooth point in the moduli space, $H^0(X, \mathcal{E}nd(E)\otimes K_X) 
\simeq H^0(X, K_X)$. 
Thus, by taking cohomology long exact sequence corresponding to the exact sequence 
(\ref{eqn4}), we get
\begin{align}\label{align1}
 \begin{split}
  0 &\to H^0(X,K_X) \to H^0(C,\mathcal{O}_C(C) \otimes A^* \otimes (E^*\otimes K_X)|_{_C})\\ 
 &\to H^1(X, (E^*\otimes K_X)^{\oplus 2}) \xrightarrow{\eta} 
 H^1(X, \mathcal{E}nd(E)\otimes K_X) \to \dots
 \end{split}
\end{align}
Note that the Petri map $\mu$ in (\ref{eqn 1}) is same as the map $\eta$ above. 
Now we will show that the map $\eta$ is injective, and this will in turn imply that the bundle 
$E$ is a smooth point in $B_2^1$.

Dualizing the exact sequence (\ref{eqn3}) and restricting to 
the curve $C$, we obtain the exact sequence (on $C$)
$$
0 \to \mathcal{O}_C(-C)\otimes A \to E^*|_{_C} \to \mathcal{O}_C^{\oplus 2} 
\to A \to 0
$$
which induces the following short exact sequence (on $C$)
\begin{equation}\label{eqn5}
 0 \to \mathcal{O}_C(-C)\otimes A \to E^*|_{_C} \to A^* \to 0.
\end{equation}
Tensoring the sequence (\ref{eqn5}) by $\mathcal{O}_C(C) \otimes A^* \otimes K_X|_{_C}$ 
and taking cohomology long exact sequence we get
\begin{align*}
 0 &\to H^0(C, K_X|_{_C}) \to H^0(C, \mathcal{O}_C(C) \otimes A^* \otimes 
 (E^*\otimes K_X)|_{_C})\\
 &\to H^0(C, \mathcal{O}_C(C) \otimes A^{*\otimes 2} \otimes K_X|_{_C}) \to \dots 
\end{align*}
By our assumption on $A$, $H^0(C, \mathcal{O}_C(C)\otimes A^*)=0$. Since 
$\text{degree}(A^*\otimes K_X|_{_C})<0$, we have $H^0(C, \mathcal{O}_C(C) \otimes 
A^{*\otimes 2} \otimes K_X|_{_C})=0$. Thus from the above exact sequence 
$H^0(C, K_X|_{_C}) \simeq H^0(C, \mathcal{O}_C(C) \otimes A^* \otimes 
(E^*\otimes K_X)|_{_C}).$ It is easy to see that $H^0(X,K_X) \simeq H^0(C, K_X|_{_C})$. 
Consequently, the map $\eta$ in (\ref{align1}) is injective.\\
We summarize the above discussion as
\begin{Theorem}\label{T3}
 $B_2^1 \subset \M(c_2), 28 \leq c_2 \leq 32,$ contains a smooth point, and hence the 
 irreducible component containing it has the expected dimension.
\end{Theorem}

\section{Non-emptiness of $B_2^0$}

Let $X$ and $\M(c_2)$ be as in the previous section. Let $\text{Hilb}^{c_2}(X)^{\text{lci}}$ 
denotes the open subscheme of the Hilbert scheme $\text{Hilb}^{c_2}(X)$ consisting of 
length $c_2$ subschemes of $X$ which are locally complete intersections. Given 
any point $Z \in \text{Hilb}^{c_2}(X)^{\text{lci}}$, we have an exact sequence
$$
0 \to \mathcal{O}_X \to E \to \mathcal{I}_Z\otimes \mathcal{O}_X(3) \to 0
$$
where $\mathcal{I}_Z$ is the ideal sheaf of $Z$ and
$E$ is a rank two torsion-free sheaf on $X$. The space of such (isomorphic classes of) 
extensions is parametrized by $\mathbb{P}\text{Ext}^1(\mathcal{I}_Z \otimes \mathcal{O}_X(3), 
\mathcal{O}_X)$. By duality, $\text{Ext}^1(\mathcal{I}_Z \otimes \mathcal{O}_X(3), 
\mathcal{O}_X) = H^1(X, \mathcal{I}_Z \otimes \mathcal{O}_X(4))^*$. Since $h^0(X, 
\mathcal{O}_X(4))=35$, for a general element of length $c_2 \geq 36$ in 
$\text{Hilb}^{c_2}(X)^{\text{lci}}$, we have $h^0(X, \mathcal{I}_Z \otimes 
\mathcal{O}_X(4))=0$. Thus a general element in $\text{Hilb}^{c_2}(X)^{\text{lci}}, 
c_2 \geq 36$, satisfies the Cayley-Bacharach property for $\mathcal{O}_X(4)$, and 
consequently, we get that the corresponding extension $E$ is locally free. 
Also any such vector bundle $E$ is $H$-stable.  
Thus for $c_2 \geq 36, B_2^0 \subset \M(c_2)$ is non-empty. 

The following two results give a bound for $\text{dim}B_2^0$.
\begin{Lemma}(\cite[Proposition 1.1]{Nijsse}\footnotemark)
 With the notations as above, for $c_2 \geq 36, \text{dim}B_2^0 
 \leq 3c_2 - 36$.
\end{Lemma}

\footnotetext{In \cite{Nijsse}, Proposition $1.1$ was proved with the assumption 
that $c_1(E)=H$, and the bound author got there is $\text{dim}B_2^0 \leq 3c_2 -11$, 
for $c_2 \geq 10$. The same calculation, in this case, gives us the stated bound}

\begin{Lemma}(\cite[Corollary 3.1]{Simpson2})
 With the notations as above, every irreducible component of  $B_2^0$ 
 has dimension $\geq 3c_2 - h^0(X, \mathcal{O}_X(3) \otimes K_X) -1.$
\end{Lemma}
Since $h^0(X, \mathcal{O}_X(3) \otimes K_X)=35,$ we see that, for $c_2 \geq 36$, 
every irreducible component of $B_2^0$ has dimension exactly $3c_2 -36$. On the 
other hand, the Brill-Noether number $\rho_2^0 = 3c_2 -36$. We summarize the 
above discussion as 
\begin{Proposition}
 With the notations as above, $B_2^0 \subset \M(c_2)$ is non-empty for 
 $c_2 \geq 36$, and every irreducible component of $B_2^0$ has the expected 
 dimension.
\end{Proposition}

{\it Acknowledgement:} We would like to thank Prof. P. Newstead for his valueable 
comments and pointing out the gap in the earlier version.
We also would like to thank D.S. Nagaraj, V. Balaji, P. Sastry for their 
encouragement and helpful discussion. First named author would like to thank 
IISER Trivandrum for their hospitality during the stay where this work started.

\end{document}